\documentclass[11pt]{article}
\usepackage{bbm}
\usepackage{geometry}
\usepackage{color}
\usepackage{amsfonts}
\usepackage{latexsym, amssymb, amsmath, amscd, amsthm, amsxtra}
\usepackage{mathtools}
\usepackage{enumerate}
\usepackage[all]{xy}
\usepackage{mathrsfs}
\usepackage{fancyhdr}
\usepackage{listings}
\usepackage{hyperref}

\def\P{\mathbb{P}}

\def\E{\mathbb{E}}

\def\R{\mathbb{R}}

\def\11{\mathbbm{1}}

\newtheorem{thm}{Theorem}[section]

\newtheorem{proposition}[thm]{Proposition}

\newtheorem{lemma}[thm]{Lemma}

\theoremstyle{definition}
\newtheorem{remark}[thm]{Remark}

\numberwithin{equation}{section}

\begin{document}
\title{Shotgun assembly threshold for lattice labeling model}

\author{Jian Ding\\Peking University \and Haoyu Liu\\Peking University}

\maketitle

\begin{abstract}
We study the shotgun assembly problem for the lattice labeling model, where i.i.d.\ uniform labels are assigned to each vertex in a $d$-dimensional box of side length $n$. We wish to recover the labeling configuration on the whole box given empirical profile of labeling configurations on all boxes of side length $r$. We determine the threshold around which there is a sharp transition from impossible to recover with probability tending to 1, to possible to recover with an efficient algorithm with probability tending to 1. Our result sharpens a constant factor in a previous work of Mossel and Ross (2019) and thus solves a question therein.
\end{abstract}

\section{Introduction}
The shotgun assembly problems of labeled graphs in general aim for recovering a global structure from local observations. This set of problems have substantial interests in applications such as DNA sequencing \cite{AMRW96, DFS94, MBT13} and recovering neural networks \cite{KPPSP13}. We learned the precise formulation and the general mathematical framework for shotgun assembly questions from the inspiring paper \cite{MR19}. Since (the circulation of) \cite{MR19}, there has been extensive study on shotgun assembly questions including on random jigsaw problems \cite{BBN17, Martinsson19, BFM20,  HLW20},  on random graph models \cite{MS15, GM22, HT21, AC22}, on random coloring model \cite{PRS22} and on some extension of DNA sequencing model \cite{RBM21}. 

In this paper we study the shotgun assembly for the lattice labeling model, whose precise mathematical formulation was proposed in \cite{MR19}. For $d\geq 1$, let $\Lambda_n = \{v\in \mathbb Z^d: |v|_\infty \leq n-1, v_i\geq 0 \mbox{ for all } 1\leq i\leq d\}$ be the box of $n^d$ vertices with the origin $o\in \mathbb Z^d$ being its smallest corner (here $|\cdot|_\infty$ denotes the $\ell_\infty$-norm of a vector, and we say $x\leq y$ if $x_i \leq y_i$ for $1\leq i\leq d$).  For $q\geq 1$, let $\sigma_v$ be i.i.d.\ labels uniformly sampled from $\{\mathsf 1, \ldots, \mathsf q\}$. For $r\geq 1$, let $\mathfrak B_{r} = \mathfrak B_{n, r}$ be the collection of all $r$-boxes (an $r$-box is a box with $r^d$ vertices) contained in $\Lambda_n$. For $B\in \mathfrak B_r$, let $\tau_B$ be the translation which maps $B$ to $\Lambda_r$. As in \cite{MR19}, we wish to recover $\{\sigma_v: v\in \Lambda_n\}$ from the empirical profile $\{\sigma|_{B}: B\in \mathfrak B_{n, r}\}$ where $\sigma|_B = \{\sigma_{\tau_B^{-1}(v)}: v\in \Lambda_r\}$. In words, our observations are labeling configurations in all the $r$-boxes  without information on locations for these $r$-boxes. (Note that in our formulation, we choose to assume that the orientation of the $r$-box is observable to us, and the similar cases when the labeling configuration is only known up to rotation/reflection symmetry can be treated by our method similarly, see Section~\ref{sec:extension}). We say that the labeling configuration is \emph{non-identifiable} if there exist two different labeling configurations on $\Lambda_n$ which would produce the same  empirical profile $\{\sigma|_{B}: B\in \mathfrak B_{n, r}\}$; otherwise we say that the labeling configuration is \emph{identifiable}.  Previously, the best result was due to \cite{MR19} which provided upper and lower bounds on the identifiability threshold up to a multiplicative constant factor. Our main contribution determines the sharp identifiability threshold, which solves \cite[Question 1.3]{MR19} (in fact, we also find the explicit formula for the threshold which was mentioned as a challenging problem in \cite{MR19}).

\begin{thm}\label{thm-main}
The following hold for any fixed $\epsilon>0$.

For $d=1$, with probability tending to 1 as $n\to \infty$ the labeling configuration is identifiable when $r \geq \frac{2(1+\epsilon)\log n}{\log q}$ and non-identifiable when $r \leq \frac{2(1-\epsilon)\log n}{\log q}$.

For $d\geq 2$, with probability tending to 1 as $n\to \infty$ the labeling configuration is identifiable when $r^d \geq \frac{d(1+\epsilon)\log n}{\log q}$ and non-identifiable when $r^d \leq \frac{d(1-\epsilon)\log n}{\log q}$.

Furthermore, in the aforementioned identifiable regimes, the labeling can be recovered by a polynomial-time algorithm. 
\end{thm}

\begin{remark}
Note that the identifiability threshold for $d=1$ was known in much more precise manner from previous works \cite{AMRW96, DFS94, MR19} (see \cite[Theorem 1]{DFS94} and \cite[Proposition 3.2]{MR19}): \cite{DFS94} follows an observation of \cite{Pavel89} that the identifiability is equivalent to the existence of a unique Eulerian path on graphs defined on vertices formed by sub-strings, and this also motivated some considerations in \cite{MR19}. 
We record the result for $d=1$ here only for completeness, and in fact we also provide a proof for non-identifiability for $d=1$ as it seems to provide some intuition that can perhaps be grasped more easily than the proof of \cite{DFS94}. An interesting question is whether some extension of Eulerian path consideration would apply in higher dimensions and how that is related to our proof method. We are not quite sure about this for the following two reasons: (1) there seem to be multiple choices in building the analogous graph in higher dimensions and it is not immediately clear to us that one version of this gives a necessary and sufficient condition for identifiability; (2) perhaps more importantly it is unclear to us that it is the most productive way to try to first formulate a sufficient and necessary condition for identifiability and then to prove whether this condition occurs or not using  probabilistic arguments. As one will see in our approach, we used two related structural properties to prove identifiability and non-identifiability and \emph{a priori} it is unclear whether these structural properties provide a necessary and sufficient condition for identifiability but it just turns out that both structural properties deeply depend on whether a typical $r$-box is unique or not and as a result this allows us to establish the sharp threshold.
\end{remark}

\begin{remark}
We see that there is a conceptual difference between $d=1$ and $d\geq 2$, which is essentially rooted in the fact that there is a phase transition for percolation when $d\geq 2$ but not when $d=1$. This point will be further manifested in our proof strategy (see discussions at the beginning of Section~\ref{sec:identifiability}).
\end{remark}

\begin{remark}
When proving non-identifiability for $d\geq 2$, we show that there exist two subsets whose original labels are $\mathsf 1$ and $\mathsf 2$ such that after swapping their labels the empirical profile for local labeling configurations on $r$-boxes remains the same (see Proposition~\ref{prop-match-empirical-profile}). We feel the second moment method we employed for the proof of Proposition~\ref{prop-match-empirical-profile} is somewhat novel and may be useful in other contexts (see discussions that follow Proposition~\ref{prop-match-empirical-profile} for more details).
\end{remark}

\begin{remark}\label{rem-extension}
Theorem~\ref{thm-main} holds in the case where two labeling configurations on a box $\Lambda_n$ are viewed the same if one can be mapped to the other by rotation and/or reflection. We discuss briefly the minor modifications required for the proof in Section~\ref{sec:extension}.
\end{remark}

\begin{remark}
We expect that our method should also apply to i.i.d.\ labels with non-uniform distribution. Further, our method may shed some light on models without independence, but the dependence seems to incur substantial challenge which our current method fails to address. We think it would be an interesting direction to consider random labels with spatial mixing such as the Ising model in high temperatures. We expect that our method would be helpful but even in this setting the challenge seems to be substantial enough for us to make any convincing guess. Another interesting future direction is to investigate the situation when observations are noisy.
\end{remark}

\noindent {\bf Acknowledgement:} we warmly thank Nathan Ross for helpful discussions.

\section{Proof of non-identifiability}\label{proof-of-non-identifiability}

\subsection{The case for $d=1$}\label{sec:1d}

We first provide the proof of non-identifiability for $d=1$, which is the (much) easier part of our main theorem. In this subsection, we assume that
\begin{equation}\label{eq-1d}
q^r \le n^{2(1 - \epsilon)} \mbox{ for an arbitrary fixed small } \epsilon>0\,.
\end{equation}

Let $I_1, \ldots, I_6$ be 6 disjoint and consecutive intervals in $\Lambda_n$, each of which has $m = \lfloor n/6 \rfloor$ vertices. For each $I_j$, let $\Gamma_j$ be a collection of $\ell = \lfloor m/r \rfloor$ disjoint intervals of  $r$ vertices in $I_j$.
We will show that with probability tending to 1, there exist $B_j \in \Gamma_j$ for $j = 1, 3, 4, 6$ such that 
\begin{equation}\label{eq-1d-B-s}
\sigma|_{B_1} = \sigma|_{B_4} \mbox{ and } \sigma|_{B_3} = \sigma|_{B_6}
\end{equation} and that 
\begin{equation}\label{eq-1d-J-J'}
\sigma|_{J[1, m]} \neq \sigma|_{J'[1, m]}
\end{equation} where $J$ is the interval strictly between $B_1$ and $B_3$ and $J'$ is the interval strictly between $B_4$ and $B_6$, and $J[1, m]$ is the initial segment of $J$ with $m$ vertices (similarly for $J'[1, m]$). We see that $J$ and $J'$ divide $\Lambda_n$ into 5 disjoint intervals, and we list them from left to right as $K_1, J, K_2, J', K_3$. Let $\tau$ be a bijection on $\Lambda_n$ so that $\tau$ maps integers $0, \ldots, n-1$ as a sequence obtained from concatenating $K_1, J', K_2, J, K_3$ (i.e., we swap $J$ with $J'$---note that this can be done even when $J, J'$ have different lengths). On the event described in our claim, we see that the labeling configuration $\sigma'$ with $\sigma'(v) = \sigma(\tau(v))$ preserves the empirical profile on $r$-boxes but $\sigma' \neq \sigma$.

It remains to prove the claim. Let $Z = \sum_{B\in \Gamma_1, B'\in \Gamma_4} \mathbf 1_{\sigma|_{B} = \sigma|_{B'} }$. Then, from straightforward computations, we have $\E Z = \ell^2 q^{- r} \geq n^{\epsilon'}$ for some $\epsilon'>0$. In addition, 
$$\E Z^2 \leq  \E Z + \ell^4  q^{- 2r} = (1+o(1)) (\E Z)^2\,,$$
where the inequality follows from the fact that $\P(\sigma|_{B} = \sigma|_{B'}, \sigma|_{\tilde B} = \sigma_{\tilde B}') = q^{- 2r}$ for $B, \tilde B \in \Gamma_1$, $B', \tilde B' \in \Gamma_4$ as long as $B \neq  \tilde B$  or $B'\neq \tilde B'$. By Chebyshev's inequality, we see that  $Z \geq n^{\epsilon'}/2$ with probability tending to 1. A similar computation applies to $\Gamma_3$ and $\Gamma_6$. So this ensures the existence of $B_1, B_3, B_4, B_6$ satisfying \eqref{eq-1d-B-s}. Finally, a simple union bound yields that with probability tending to 1 for any disjoint intervals $K$ and $K'$ which contain $I_2$ and $I_5$ respectively, we have $\sigma|_{K[1,m]} \neq \sigma|_{K'[1, m]}$ (this is because $\P(\sigma|_{K[1,m]} = \sigma|_{K'[1, m]}) \leq e^{-cn}$ for some constant $c>0$ and the number of choices for such $K, K'$ is only polynomial in $n$). This verifies \eqref{eq-1d-J-J'} and thus completes the proof.

\subsection{The case for $d\geq 2$}

In this subsection we consider the non-identifiable regime where $d \ge 2$ and
\begin{equation}\label{eq-assume-subcritical}
q^{r^d} \leq n^{d(1 - \epsilon)} \mbox{ for an arbitrary fixed small } \epsilon>0\,.
\end{equation}
For $v \in \Lambda'_n = \{u\in \Lambda_n: r\leq u_i \leq n-r \mbox{ for all } 1\leq i\leq d\}$, let $\mathfrak R_v = \{R_1(v), \ldots, R_{r^d}(v)\}$ be the collection of all $r$-boxes containing $v$. We also set the notation so that the relative location of $v$ in $R_j(v)$ is the same as the relative location of $u$ in $R_j(u)$ for all $u, v\in \Lambda'_n$. For $U\subset \Lambda'_n$, we denote $\mathcal L(U) = (\mathcal L_1(U), \ldots, \mathcal L_{r^d}(U))$ where $\mathcal L_{j}(U)$ is the empirical distribution for $\{\sigma|_{R_j(u) \setminus u}: u\in U\}$. We view $\mathcal L_{j}(U)$ as a $q^{r^d-1}$-dimensional vector where its $s$-th coordinate $\mathcal L_{j, s}(U)$ counts the occurances of the $s$-th configuration (the ordering of the configurations is arbitrary but prefixed).
Denote $V_{\mathsf k} = \{v\in \Lambda'_n: \sigma_v = \mathsf k\}$ for $\mathsf k \in \{\mathsf 1, \ldots, \mathsf q\}$.
\begin{proposition}\label{prop-match-empirical-profile}
Under the assumption \eqref{eq-assume-subcritical} with probability tending to 1 as $n \to \infty$, there exist $V'_{\mathsf 1} \subset ((2r)\mathbb Z^d) \cap V_{\mathsf 1}, V'_{\mathsf 2}\subset ((2r)\mathbb Z^d)\cap V_{\mathsf 2}$ such that $\mathcal L(V'_{\mathsf 1}) = \mathcal L(V'_{\mathsf 2})$.
\end{proposition}
Proposition \ref{prop-match-empirical-profile} readily implies the non-identifiability since we can swap the labels of $V'_{\mathsf 1}$ and $V'_{\mathsf 2}$ without changing $\{\sigma|_{B}: B\in \mathfrak B_{n, r}\}$. So our main goal in this subsection is to prove Proposition~\ref{prop-match-empirical-profile}. Usually in order to prove the existence of such pair of sets one first shows that the first moment is large (as implied by Lemma~\ref{lem-match-probability} below) and then one needs to employ a second moment method. However, the implementation of a standard second moment method or even a second moment method with truncation would be quite challenging (if possible at all). This is because in our case by a first moment computation the size of desirable $V'_{\mathsf 1}$ and $V'_{\mathsf 2}$ has to be larger than $n^{2 - \epsilon'}$ for some $\epsilon'>0$, and as a result a typical pair of sets would have significant overlap which results in significant amount of correlation. However, we manage to get around this challenge since we are flexible with the size of $V'_{\mathsf 1}$ and $V'_{\mathsf 2}$, that is, we only need to show for some $M'$ (but not a fixed $M'$) there exists a desired pair of $V'_{\mathsf 1}$ and $V'_{\mathsf 2}$ of size $M'$. The key novelty in our proof lies in the definition of $\chi_{\mathbf L}$ and $\mathcal E_{\mathsf k}(U, \mathbf L)$ (see \eqref{eq-chi-mathbf-L}). We first show in Lemma~\ref{lem-force-overlap} that for a typical $\mathbf L$ we must have $\chi_{\mathbf L} = 1-o(1)$ since otherwise we would have the second moment of a random variable smaller than the square of its first moment. Once we show  $\chi_{\mathbf L} = 1-o(1)$, it is straightforward to derive Proposition~\ref{prop-match-empirical-profile}. Next, we carry out the proof details according to this outline.

For any $m\geq 1$ and $U\subset \Lambda'_n$ with $|U| = m$,  we let $\mathfrak L(U)$ be the space of all realizations for $\mathcal L(U)$. In addition, we assume that $U\subset_{2r} \Lambda'_n$, i.e., $U \subset \Lambda'_n$ is a set which has pairwise $\ell^\infty$-distance at least $2r$ (so $R_j(v)$ and $R_{j'}(u)$ are disjoint for different $u,v \in U$ and for all $1\leq j, j'\leq r^d$). In this way, the law of $\mathcal L(U)$ does not depend on the particular choice of $U$ except through $|U|$, and thus we can write $\mathfrak L(U) = \mathfrak L(|U|) = \mathfrak L(m)$ for simplicity. We write $\mu_U = \mu_{|U|}$ for the probability measure of $\mathcal L(U)$ on $\mathfrak L(U)$. For any $\iota >0$, we define
\begin{equation}\label{eq-def-chi-m}
\mathfrak L(U, \iota) = \mathfrak L(|U|, \iota) = \{\mathbf L\in \mathfrak L(U): \mu_U(\mathbf L) \geq \iota\}\,.
\end{equation}
For $\delta = \epsilon/100$, we let $M = n^{d(1 - \delta)}$. The very basic intuition behind Proposition~\ref{prop-match-empirical-profile} is encapsulated in the following lemma, since it implies heuristically that most of $\mathbf L\in \mathfrak L(M)$ should appear.
\begin{lemma}\label{lem-match-probability}
For $\iota_* = e^{-n^{d(1-\epsilon/2)}}$, we have $\mu_M(\mathfrak L(M, \iota_*)) = 1-o(1)$.
\end{lemma}
\begin{proof}
For $U\subset_{2r} \Lambda'_n$ of cardinality $M$, we see that $0\leq \mathcal L_{j, s}(U) \leq M$ for all $1\leq j\leq r^d, 1\leq s\leq q^{r^d-1}$ and thus $|\mathfrak L(M)| \leq (M+1)^{r^d q^{r^d-1}} \stackrel{\Delta}{=} K$. Therefore, 
\[ \mu_M(\mathfrak L(M) \setminus \mathfrak L(M, \iota_*)) \leq \iota_* K\,, \]
and it remains to check that $\iota_*=o(1/K)$.
Simple algebraic manipulations yield that
$$\log K=r^d q^{r^d-1} O(\log n)
\le O(1) n^{d(1-9\epsilon/10)} = o(\log \iota_*^{-1})\,,$$
as required.
\end{proof}
For $U_{\mathsf k} = (2r)\mathbb Z^d\cap V_{\mathsf k}$, we have that $|U_{\mathsf k}|$ is a binomial random variable where the number of trials is at least $((n-2r)/2r)^d$ and the success probability is $1/q$. Thus, applying concentration inequality for binomial variables we get that
\begin{equation}\label{eq-typical-U-k}
\P(|U_{\mathsf k}| \geq N \mbox{ for all } \mathsf k = \mathsf 1,\ldots, \mathsf q) \geq 1-n^{-4} \,, \mbox{ where } N  \stackrel{\Delta}{=} \frac{n^d}{q(16r)^d}\,.
\end{equation} 
Without loss of generality we can assume that $|U_{\mathsf k}| = N$ for $\mathsf k = \mathsf 1, \ldots, \mathsf q$  because if not, we can simply take a subset with cardinality $N$ and name it as $U_{\mathsf k}$. For $U\subset U_{\mathsf k}$ and $\mathbf L\in \mathfrak L(U)$, define
\begin{equation}\label{eq-chi-mathbf-L}
\chi_{\mathbf L} = \P(\mathcal E_{\mathsf k}(U, \mathbf L)\mid \mathcal L(U) = \mathbf L)\,,
\end{equation}
where
\[\mathcal E_{\mathsf k}(U, \mathbf L) =  \bigcup_{W: W \subset U_{\mathsf k}, W \neq U, W\cap U \neq \emptyset, |W| = |U|} \{\mathcal L(W) = \mathbf L\}\,. \]
Note that on $\mathcal E_{\mathsf k}(U, \mathbf L)$ and $\mathcal L(U)=\mathbf L$, there exist $\emptyset \neq U'\subset U$ and $W'\subset U_{\mathsf k}\setminus U$ such that $\mathcal L(U') = \mathcal L(W')$ (we can simply take $U' = U \setminus W$ and $W' = W \setminus U$ for an arbitrary $W$ that certifies $\mathcal E_{\mathsf k}(U, \mathbf L)$). Since we have assumed that $|U_{\mathsf k}| = N$ for all $\mathsf k$ and $\mathcal L(W')$ does not depend on $\sigma|_{W'}$ for $W' \subset_{2r} \Lambda'_n$, we see that the law of the empirical profiles $\{\mathcal L(W'): W' \subset U_{\mathsf k'}\}$ does not depend on $\mathsf k'$.  In addition, for any $\mathsf k' \neq \mathsf k$, we have that conditioned on $\mathcal L(U) = \mathbf L$, there is a coupling such that  $\{\mathcal L(W'): W' \subset U_{\mathsf k'}\} \supset \{\mathcal L(W'): W' \subset U_{\mathsf k} \setminus U\}$. Therefore,
\begin{equation}\label{eq-lower-bound-prob-match}
\P(\cup_{\emptyset \neq U'\subset U, W' \subset U_{\mathsf k'}}\{\mathcal L(W') = \mathcal L(U')\} \mid \mathcal L(U) = \mathbf L) \geq \chi_{\mathbf L}\,.
\end{equation}

\begin{lemma}\label{lem-force-overlap}
For each $U\in U_{\mathsf k}$ with $|U| = M$ and $\mathbf L \in \mathfrak L(U, \iota_*)$, we have $\chi_{\mathbf L} \geq 1- n^{-1}$.
\end{lemma}
\begin{proof}
Suppose otherwise there exist $\mathbf L \in \mathfrak L(U, \iota_*)$ with $\chi_{\mathbf L} \leq 1 - n^{-1}$. Define 
\[ Z_{\mathbf L} = \sum_{U \subset U_{\mathsf k}: |U| = M} \mathbf 1\{\mathcal L(U) = \mathbf L; (\mathcal E_{\mathsf k}(U, \mathbf L))^c\}\,. \]
Since $\chi_{\mathbf L} \leq 1 - n^{-1}$, we see that
\begin{align}
\E Z_{\mathbf L} &\ge \binom{N}{M} \mu_M(\mathbf L) n^{-1} \ge \binom{N}{M} \iota_* n^{-1} \nonumber\\
&\ge \left( \frac{N-M}{M} \right)^M e^{-n^{(1-\epsilon/2)d}} n^{-1} \nonumber \\
&= \left(\frac{n^d/(q(16r)^d)-n^{d(1-\delta)}}{n^{d(1-\delta)}} \right)^{n^{d(1-\delta)}} e^{-n^{(1-\epsilon/2)d}} n^{-1} \nonumber \\
&\ge n^{(d \delta/2) \cdot n^{d(1-\delta)}} e^{-n^{(1-\epsilon/2)d}} n^{-1} \gg n^6\,, \label{eq-Z-L-first-moment-lower}
\end{align}
where the second inequality follows from $\mathbf L\in \mathfrak L(U, \iota_*)$.
In addition, we can compute its second moment as 
\begin{align}
\E Z_{\mathbf L}^2 &= \sum_{U, U' \subset U_{\mathsf k}: |U| = |U'| = M}  \E \mathbf 1\{\mathcal L(U) = \mathbf L; (\mathcal E_{\mathsf k}(U, \mathbf L))^c\}  \mathbf 1\{\mathcal L(U') = \mathbf L; (\mathcal E_{\mathsf k}(U', \mathbf L))^c\} \nonumber\\
&\leq \sum_{U, U' \subset U_{\mathsf k}: |U| = |U'| = M, U\cap U' = \emptyset} \E \mathbf 1\{\mathcal L(U) = \mathbf L\}  \mathbf 1\{\mathcal L(U')  = \mathbf L\} +\sum_{U \subset U_{\mathsf k}: |U|= M} \E \mathbf 1\{\mathcal L(U) = \mathbf L\}  \nonumber \\
&\leq \binom{N}{M} \binom{N-M}{M} (\mu_M(\mathbf L))^2 + \binom{N}{M} \mu_M(\mathbf L)\,, \label{eq-Z-L-second-upper}
\end{align}
where the first inequality follows since on $(\mathcal E_{\mathsf k}(U, \mathbf L))^c$ for any legitimate $U'$ with $U'\cap U \neq \emptyset$ and $U' \neq U$ we have $\mathcal L(U') \neq \mathbf L$. 
Since $\binom{N}{M} \mu_M(\mathbf L)\geq n^7$ and since
\[ \binom{N}{M}/\binom{N-M}{M} \ge \left(\frac{N}{N-M} \right)^M \gg n^2\,,\]  we have that
\[ \binom{N-M}{M} \mu_M(\mathbf L)+1 \ll \binom{N}{M} \mu_M(\mathbf L) n^{-2}\,.\]
Combined with \eqref{eq-Z-L-first-moment-lower} and \eqref{eq-Z-L-second-upper}, it yields that $\E Z_{\mathbf L}^2 \ll (\E Z_{\mathbf L})^2$, arriving at a contradiction and thereby concluding the proof of the lemma.
\end{proof}
\begin{proof}[Proof of Proposition~\ref{prop-match-empirical-profile}]
Take a $U\subset U_{\mathsf 1}$ with $|U| = M$. By Lemma~\ref{lem-match-probability}, with probability $1-o(1)$ we have that $\mathcal L(U) \in \mathfrak L(U, \iota_*)$. By \eqref{eq-lower-bound-prob-match} and Lemma~\ref{lem-force-overlap}, we see that with probability  $1 - o(1)$, there exist $\emptyset \neq U'\subset U$ and $W' \subset U_{\mathsf 2}$ such that $\mathcal L(U') = \mathcal L(W')$. This completes the proof of the proposition.
\end{proof}

\section{Proof of identifiability}\label{sec:identifiability}

In this section we prove identifiability for $d \geq 2$. Recall that in this regime
\begin{equation}\label{eq-assume-supercritical}
q^{r^{d}} \geq n^{d(1 + \epsilon)} \mbox{ for an arbitrary fixed small } \epsilon>0\,.
\end{equation}
We note that the threshold is chosen as in \eqref{eq-assume-supercritical} since this ensures that for  each $B\in \mathfrak B_r$ we have that $\sigma_B$ is unique with probability tending to 1 (see Lemma~\ref{lem-unique-typical})---this is a property we will repeatedly use in our proof.

Our recovering procedure will employ the following three steps to successively \emph{determine} labels on vertices of $\Lambda_n$ (here we say that we determine $\sigma_v = \mathsf k$ if every $\sigma$ with given empirical profile satisfies $\sigma_v = \mathsf k$):\\
\noindent {\bf Step 1: initial labeling at the corner.} Determine labels on $\Lambda_{2r}$.\\
\noindent {\bf Step 2: percolation of unique $(r-1)$-boxes.} Inductively check each \emph{unexplored} box in $\mathfrak B_{r-1}$ (denoted as $B$) where all labels have been previously determined. If $\sigma|_B$ is unique over all $(r-1)$-boxes, we can then find a few boxes $B'\in \mathfrak B_{r}$ so that $\sigma|_{B'}$ agrees with $\sigma|_B$ on a sub-$(r-1)$ box of $B'$ (we can find $2^d$ such $B'$'s unless $B$ is near the boundary of $\Lambda_n$). Therefore, from $\sigma|_{B'}$ we can determine labels on neighboring vertices of $B$ (they may have been determined already) and then we mark $B$ as \emph{explored}.\\
\noindent {\bf Step 3: final step of recovery.} For each vertex $v$ which is not determined after \textbf{Step 2}, check all boxes $B\in \mathfrak B_r$  containing $v$. For each such $B$, let $\mathsf B\subset B$ be the collection of determined vertices in $B$. If $\sigma|_{\mathsf B}$ is unique over all translated copies of $\mathsf B$ (note that this can be checked by scanning through $\{\sigma|_B: B\in \mathfrak B_r\}$), we can then determine labels on $B$ and thus in particular the label on
$v$.\\
In order to justify correctness for each aforementioned step we will use probabilistic arguments to prove some desirable structural properties which hold for a typical labeling configuration. These properties, altogether, will show that eventually we have determined labels on all vertices of $\Lambda_n$. In addition, it is easy from the description of our procedure that the running time is polynomial in $n$.

Next, we provide a proof for correctness of our 3-step procedure while omitting proofs for a few lemmas and propositions (in this way of exposition we hope that this can serve as an overview for our proof that helps a reader to grasp the proof sketch before jumping into details). To this end, we first introduce a few terminologies. For each set $B$, we say $B$ is \emph{unique} if $\sigma|_B$ is unique among $\{\sigma|_{B'}: B' \subset \Lambda_n \mbox{ is a translated copy of } B\}$. In particular,  for each $B \in \mathfrak B_{r-1}$, we say  $B$ is \emph{unique} if $\sigma|_B$ is unique in $\{\sigma|_{B'}: B'\in \mathfrak B_{r-1}\}$. For each $B\in \mathfrak B_{s}$, we say that $B$ is \emph{open} if  each $B'\in \mathfrak B_{r-1}$ that is contained in $B$ is unique. The next lemma formalized the intuition for the choice of threshold in \eqref{eq-assume-supercritical}.
\begin{lemma}\label{lem-unique-typical}
For each $B\in \mathfrak B_{2r}$, we have $\P(B \mbox{ is open}) \geq 1 - n^{-\epsilon/2}$.
\end{lemma}
Uniqueness is useful due to the following lemma.
\begin{lemma}\label{lem-successively-determine}
For $s \geq r$, if $B\in \mathfrak B_s$ is open and labels on an $(r-1)$-sub-box of $B$ are determined, then there is a polynomial time algorithm which determines labels on $B$.
\end{lemma}
We are now ready to prove the correctness of {\bf Step 1}, as formulated in the next proposition.
\begin{proposition}\label{prop-Step-1}
On the event that $\Lambda_{2r}$ is open (which happens with probability tending to 1 by Lemma~\ref{lem-unique-typical}), we can determine labels on $\Lambda_{2r}$.
\end{proposition}

We now turn to {\bf Step 2}.  Let $\mathcal B_{2r}$ be the collection of boxes in $\mathfrak B_{2r}$ where each coordinate of the largest corner vertex is  either equal to $n$ or of the form $kr + (r-1)$ for some integer $k$. Essentially, $\mathcal B_{2r}$ is a disjoint partition of $\Lambda_n$ into $(2r)$-boxes together with $(2^{d}-1)$ shifts of the partition (where each coordinate is either shifted by $r$ or not shifted). For $B, B'\in \mathcal B_{2r}$, we say that $B$ is strongly neighboring to $B'$ if $|B \cap B'| \geq r^d$ and we say $B$ is weakly neighboring to $B'$ if $ \min_{u\in B, v\in B'}|u-v|_{\infty} \leq 4r$. Let $\mathcal C_{2r}$ be the collection of $B\in \mathcal B_{2r}$ that is connected to $\Lambda_{2r}$ via a sequence of strongly neighboring and open boxes in $\mathcal B_{2r}$. We say $B$ is weakly connected to $B'$ if there is a weakly neighboring sequence of boxes in $\mathcal B_{2r}$ joining $B$ and $B'$.  The following percolation type of result is the key input for analyzing {\bf Step 2}.
\begin{proposition}\label{prop-Step-2}
With probability tending to 1, each weakly connected component in $\mathcal B_{2r} \setminus \mathcal C_{2r}$ has diameter at most $\kappa r$ where $\kappa = \kappa(d, \epsilon)$ is independent of $n$. In addition, for any $\kappa' = \kappa'(d, \epsilon)$ with probability tending to 1 we have that 
$$\mathrm{Corner} \stackrel{\Delta}{=}\{v\in \Lambda_n: v_i\not\in (\kappa' r, n- \kappa' r) \mbox{ for } 1\leq i\leq d\} \subset \mathcal C_{2r}\,.$$
\end{proposition}
By Lemma~\ref{lem-successively-determine}, each vertex in $\mathcal C_{2r}$ was determined in {\bf Step 2} and thus by Proposition~\ref{prop-Step-2}, 
\begin{equation}\label{eq-step-2-good-property}
\mathcal B_{2r}\setminus \mathcal C_{2r}  \mbox{ consists of weakly connected components with diameter } \leq \kappa r\,.
\end{equation} 
This will be a very useful input for {\bf Step 3}, as implied by the next proposition. For $v\in \Lambda_N$ with $v_1,v_2,\ldots,v_d \leq n-r$ and $v_2 \geq r$, for $s =0, \ldots, r-1$ define
\begin{equation}\label{eq-def-B-k-v}
B_s(v) = (v_1 + 1, v_2 - s, v_3, \ldots, v_d) + [0, r-2] \times [0, r-1]^{d-1}
\end{equation}
to be a collection of rectangles (in fact each rectangle is almost an $r$-box) containing $v$ which are fully contained in $\Lambda_n$. (We made this particular choice so that in our application later all vertices in $B_k(v)$ have been determined in {\bf Step 2}.)
\begin{proposition}\label{prop-Step-3}
We have that
$$\P(\exists v\in \Lambda_n: B_s(v) \mbox{ is not unique for all } 0\leq s\leq r-1) = o(1)\,.$$
\end{proposition}
We now explain how Proposition~\ref{prop-Step-3} ensures typically all vertices are determined at the end of {\bf Step 3}. Suppose otherwise there exists a vertex $v$ that is not determined. By \eqref{eq-step-2-good-property}, $v$ is in a weakly connected component of undetermined vertices whose diameter is at most $\kappa r$ and in addition this component is disjoint from $\mathrm{Corner}$ (in our application, when defining $\mathrm{Corner}$ we take $\kappa' = \kappa + 2$). Thus, for each component there exists a vertex $v$ so that one of its coordinates is in $(\kappa' r, n - \kappa'r)$ and we may assume without loss of generality that the second coordinate of $v$ is in $(\kappa'r, n- \kappa' r)$. Since the component has diameter at most $\kappa r$, by our choice of $\kappa'$ we see that the second coordinates for all vertices in this component are  in $( r, n- r)$. We further assume that in this component $v$ has the largest first coordinate and $v_1,v_2,\ldots,v_d \leq n-r$ and $v_2 \geq r$ (the analysis is completely the same by symmetry in other cases). Since $v$ has the largest first coordinate in this component and since all components have mutual $\ell_\infty$-distance at least $4r$ (by our definition of weakly neighboring), we see that all vertices in $\cup_{s=0}^{r-1}B_s(v)$ have been determined. By Proposition~\ref{prop-Step-3}, we see that with probability $1-o(1)$ we have that for all $u$ there exists $B_{s_u}(u)$ which is unique. We then assume without loss of generality that this event occurs. Since vertices in  $B_{s_v}(v)$ have all been determined, we can then check sequentially for $s = 0, \ldots, r-1$ and for each such $s$ we scan through the empirical profile $\{\sigma|_B: B\in \mathfrak B_r\}$ until we find the first $s_v$ satisfying the following property: there is a \emph{unique} $\tilde \sigma \in \{\sigma|_B: B\in \mathfrak B_r\}$ such that when viewed as a labeling configuration on $\Lambda_r$ the labeling configuration of
$\tilde \sigma|_{\Lambda_r \setminus \{ x\in \Lambda_r: x_1 = 0\}}$ agrees with $\sigma|_{B_{s_v}(v)}$.
At this point, from the value of $s_v$ and $\tilde \sigma$ we can determine the label on $v$. This arrives at a contradiction and thus completes the proof of our theorem.

Next, we provide proofs for omitted lemmas and propositions, which are organized into three subsections corresponding to the three steps in our procedure.

\subsection{Proofs for Step 1}

In this subsection, we provide proofs for Lemmas~\ref{lem-unique-typical}, \ref{lem-successively-determine} and Proposition~\ref{prop-Step-1} in order.

\begin{proof}[Proof of Lemma~\ref{lem-unique-typical}]
For any $B\in \mathcal B_{r-1}$, we have
\begin{equation}\label{eq-unique-probability}
\P(B \mbox{ is not unique}) \leq \sum_{B'\in \mathfrak B_{r-1}, B'\neq B} \P(\sigma|_B = \sigma|_{B'}) \leq n^d q^{-(r-1)^{d}}\,,
\end{equation}
where the last inequality follows from \eqref{eq-assume-supercritical}. Therefore, for $B\in \mathfrak B_{2r}$,
$$\P(B \mbox{ is not open }) \leq \sum_{B'\subset B: B'\in \mathfrak B_{r-1}} \P(B' \mbox{ is not unique}) \leq (r+2)^d n^d q^{-(r-1)^{d}}  \leq n^{-\epsilon/2}\,,$$
completing the proof of the lemma.
\end{proof}

\begin{proof}[Proof of Lemma~\ref{lem-successively-determine}]
The proof is similar to and follows that of \cite[Lemma 2.3, Proposition 3.2]{MR19}. It suffices to show that on the event that $B$ is open,  we can determine the label for $v\in B$ if $v$ is neighboring to an $(r-1)$-box $A$ whose vertices have all been determined (and for convenience we assume that $v$ is neighboring to the surface of $A$ with largest first coordinate). To this end, we scan through $\{\sigma|_{B'}: B'\in \mathfrak B_r\}$ and let $\tilde \sigma$ (viewed as a labeling configuration on $\Lambda_{r}$) be the first labeling configuration so that
$\tilde \sigma|_{\{w\in \Lambda_r: w_1 \leq r-2\}} = \sigma|_A$.
Since $A$ is unique by our assumption, such translated copy in $\tilde \sigma$ is unique. As a result, this induces a unique  mapping  $\tau: \Lambda_r \mapsto B$ so that $\tilde \sigma(u) = \sigma(\tau(u))$ for all $u\in \Lambda_r$. By our assumption that  $v$ is neighboring to the surface of $A$ with largest first coordinate, we see that $v\in \tau(\Lambda_r)$ and as a result we can determine the label on $v$.
\end{proof}

\begin{proof}[Proof of Proposition~\ref{prop-Step-1}]
By Lemma~\ref{lem-successively-determine} (and since we assume that $\Lambda_{2r}$ is open), it suffices to show that we can determine labels on $\Lambda_{r-1}$. For convenience, for each $B \in \mathfrak B_r$, we let $B_1,\ldots,B_{2^d}$ be the $(r-1)$-sub-boxes of $B$ that are contained in $B$, and $B_1$ is the one that contains the smallest vertex in $B$. Under the assumption of uniqueness of $\Lambda_r$, we see that $\sigma|_{\Lambda_{r-1}}$ is the unique labeling configuration so that there exists $B \in \mathfrak B_r$ with $\sigma|_{\Lambda_{r-1}}=\sigma|_{B_1}$ but there is no $B \in \mathfrak B_r$ with $\sigma|_{\Lambda_{r-1}}=\sigma|_{B_i}$ for any $i=2,\ldots,2^d$.
\end{proof}

\subsection{Proofs for Step 2}

This subsection is devoted to the proof of Proposition~\ref{prop-Step-2}, which closely resembles the now standard coarse graining method widely used in percolation theory. We say a box in $\mathfrak B_{2r}$ is closed if it is not open. For sets $A, A', S\subset \Lambda_n$, we say that $A$ is weakly separated from $A'$ by $S$ if any path from $A$ to $A'$ has to go through a vertex whose $\ell_\infty$-distance to $S$ is at most $4r$. Similarly, we say that $A$ is weakly enclosed by $S$ if any path from $A$ to $\Lambda_n^c$ has to go through a vertex whose $\ell_\infty$-distance from $S$ is at most $4r$. Write $\partial_{4r} \Lambda_n = \{u\not\in \Lambda_n: |u-v|_{\infty} = 4r \mbox{ for some } v\in \Lambda_n\}$. Let $\mathcal S$ be the collection of components generated by weakly connected closed boxes in $\mathcal B_{2r}$. Note that $S \in \mathcal S$ is a collection of boxes in $\mathcal B_{2r}$, and we denote its vertex set by $V(S) = \cup_{B\in S} B$. 
By duality considerations, we see that if $\mathcal B_{2r} \setminus \mathcal C_{2r}$ has a weakly connected component with diameter $> \kappa r$ then there exists $v\in \Lambda_n \setminus \Lambda_{2r}$ and a weakly connected closed component $S\in \mathcal S$ such that $V(S)\cup \partial_{4r} \Lambda_n$ weakly separates $\Lambda_{2r}$ from $v$ and that either
\begin{equation}\label{eq-enclose-Lambda-2r}
|S| \geq \frac{\min_{u\in \Lambda_{2r}, w\in V(S)} |u-w|_{\infty}}{100r}
\end{equation} 
or 
\begin{equation}\label{eq-enclose-v}
|S| \geq  \frac{\min_{ w\in V(S)} |v-w|_{\infty}}{100r} \mbox{ and } |S| \geq \kappa/100^d\,.
\end{equation}
(In spirit, in the above \eqref{eq-enclose-Lambda-2r} corresponds to the case that $\Lambda_{2r}$ ``does not percolate'' and \eqref{eq-enclose-v} corresponds to the case that there is big weakly connected closed component although $\Lambda_{2r}$ may percolate globally.)

The following lemma is a key input in order to upper-bound probabilities for either of the two cases.
\begin{lemma}\label{lem-S-closed}
For any $B\in \mathcal B_{2r}$, we have that 
$$\P(\exists S \in \mathcal S: B\in S \mbox{ and } |S| \geq t) \leq n^{-8^{-d}\epsilon t}\,.$$
\end{lemma}
\begin{proof}
If there exists $S \in \mathcal S$ with $B\in S$  and $|S| \geq t$, then we claim that there exists a collection of \emph{disjoint} and closed $2r$-boxes $\mathsf S$ with $|\mathsf S| \geq 4^{-d} t$ such that $\mathsf S$ is $8r$-weakly-connected (here $8r$-weakly connected corresponds to the $8r$-weakly-neighboring for two boxes which means that the $\ell_\infty$-distance between these two boxes is at most $8r$). In order to see this, we can for instance take a maximal $\mathsf S \subset S$ with $B\in \mathsf S$ such that all sets in $\mathsf S$ are mutually disjoint.  Since such $\mathsf S$ with $|\mathsf S| = t'$ induces at least one $8r$-weakly-connected tree on $\mathcal B_{2r}$ of size $t'$, and each such tree can be encoded by its depth-first-search contour started with $B$ and of length $2t'$, we can then upper-bound the number of choices for such $\mathsf S$ with $|\mathsf S| = t'$ by the number of $8r$-weakly-connected paths on $\mathcal B_{2r}$ started with $B$ and of length $2t'$. That is, the enumeration is bounded by $16^{2dt'}$. 

We next wish to upper-bound the  probability for all boxes in $\mathsf S$ being closed.  Due to disjointness, we may wish to bound this by $n^{-|\mathsf S|\epsilon/2}$ in light of Lemma~\ref{lem-unique-typical}. This is not completely correct since even for disjoint boxes their openness are not exactly independent, although the fix is easy as we explain next. 
For each $B \in \mathsf S$, since $B$ is closed then there exists an $(r-1)$-box $B' \subset B$ such that $\sigma|_{B'} = \sigma|_{B''}$ for some $B''\in \mathfrak B_{r-1}$, in which case we draw an edge between $B'$ and $B''$. In this way, we can draw an edge from each $B\in \mathsf S$ and we let $\mathcal E$ be the collection of all the edges. We can then take a subset of such  edges $\mathcal E$ with  $|\mathcal E| = t''= \lfloor|\mathsf S|/2\rfloor$ so that there is no cycle among these edges (there is no cycle even when edges are viewed as edges between sets in $\mathsf S$, that is, even when each edge between $B' \subset B$ and $\tilde B' \subset \tilde B$ for $B, \tilde B\in \mathsf S$ is viewed as an edge between $B$ and $\tilde B$). On the one hand, the number of labeling configurations that are consistent with $\mathcal E$ is at most $q^{|\Lambda_n|} q^{-(r-1)^d |\mathcal E|} = q^{|\Lambda_n|} q^{-(r-1)^d t''}$ (note that the acyclic property here ensures that each edge in $\mathcal E$ reduces the number of consistent labeling configurations by a factor of $q^{(r-1)^d}$). On the other hand, the number of possible choices for $\mathcal E$ is at most $2^{|\mathsf S|} ((r+1)^d)^{t''} (n^d)^{t''}$. Putting this together, we see that
\begin{equation*}
\P(\mbox{ all boxes in $\mathsf S$ are closed}) \leq  \frac{q^{|\Lambda_n|} q^{-(r-1)^d t''}}{ q^{|\Lambda_n|}}  2^{|\mathsf S|} ((r+1)^d)^{t''} (n^d)^{t''} \leq n^{-\epsilon |\mathsf S|/4}\,.
\end{equation*}
Combined with the aforementioned upper bound on the enumeration for $\mathsf S$, we get that
$$\P(\exists S \in \mathcal S: B\in S \mbox{ and } |S| \geq t)  \leq \sum_{t_1 \ge t}  \sum_{t' \geq 4^{-d}t_1} 16^{2dt'} n^{-\epsilon t'/4} \leq n^{-8^{-d}\epsilon t}\,,$$
as required.
\end{proof}

We are now ready to provide
\begin{proof}[Proof of Proposition~\ref{prop-Step-2}]
We first treat the case as in \eqref{eq-enclose-Lambda-2r}. In light of Lemma~\ref{lem-S-closed}, it suffices to sum over all choices for the ``starting'' box  $B$. Let $t$ be the right hand side of \eqref{eq-enclose-Lambda-2r}. Then the number of choices for $B\in \mathcal B_{2r}$ with $\ell_\infty$-distance at most $100rt$ to $\Lambda_{2r}$ is at most $(100t)^{d}$. Therefore,
\begin{equation}\label{eq-prob-enclose-Lambda-r}
\P(\exists S\in \mathcal S: \eqref{eq-enclose-Lambda-2r} ~ holds) \leq \sum_{t\geq 1}\sum_{t'\geq t} n^{-8^{-d}t'} (100t)^{d}  = o(1)\,.
\end{equation}
Similarly we can bound the case for \eqref{eq-enclose-v}. In this case, there is a slight difference in bounding the enumeration for the ``starting'' box  $B$: the vertex $v$ can be chosen arbitrarily and as a result the number of choices for $B$ is at most $n^d$.  Therefore, another application of Lemma~\ref{lem-S-closed} gives that
$$\P(\exists S\in \mathcal S: \eqref{eq-enclose-v} ~ holds) \leq \sum_{t\geq \kappa/100^d} n^{-8^{-d}\epsilon t} n^d  = o(1) $$
as long as $\kappa = \kappa (d, \epsilon)$ is a large enough constant. Combined with \eqref{eq-prob-enclose-Lambda-r}, this completes the proof of the first part of the proposition.

Finally, we show that $\mathrm{Corner} \subset \mathcal C_{2r}$. This follows since otherwise for some weakly connected closed component $S$ we have that $|S| \geq \frac{\min_{w\in V(S)}|v-w|_\infty}{100r}$ (i.e., \eqref{eq-enclose-v} without the part $|S| \geq \kappa/100^d$) for some $v\in \mathrm{Corner}$. Since $|\mathrm{Corner}| = O((\kappa' r)^d)$, a similar computation as in \eqref{eq-prob-enclose-Lambda-r} shows that this happens with probability $o(1)$.
\end{proof}

\subsection{Proofs for Step 3}

In this subsection, we prove Proposition~\ref{prop-Step-3}. For sets $A, A' \subset \Lambda_n$, if $A$ is a translated copy of $A'$, we then define $\mathsf v_{A, A'}$ as the unique vector such that $A' = \{v + \mathsf v_{A, A'}: v\in A\}$. Recall \eqref{eq-def-B-k-v}. If for $v\in \Lambda_N$ with $v_1, v_2, \ldots,v_d \leq n-r$ and $v_2 \geq r$ we have that $B_s(v)$ is not unique for all $s = 0, \ldots, r-1$. Then, recursively for $s = 0, \ldots, r-1$ we can pick a $\mathsf B'_s \neq B_s(v)$ such that $\sigma|_{\mathsf B'_s} = \sigma|_{B_s(v)}$, and further we pick $\mathsf B'_s$ such that $\mathsf v_{B_{s-1}(v), \mathsf B'_{s-1}} = \mathsf v_{B_{s}(v), \mathsf B'_{s}}$ for $s\geq 1$ if this is possible. Note that for $s_1<s_2$ if it is possible to pick $\mathsf B'_{s_1}$ and $\mathsf B'_{s_2}$ such that $\mathsf v_{B_{s_1}(v), \mathsf B'_{s_1}} = \mathsf v_{B_{s_2}(v), \mathsf B'_{s_2}}$ then it would have been possible and by our rule we would have picked $\mathsf B'_{s}$ for $s_1 < s<s_2$ such that $\mathsf v_{B_{s_1}(v), \mathsf B'_{s_1}} = \mathsf v_{B_{s}(v), \mathsf B'_{s}}$. Therefore, the interval $\{0, \ldots, r-1\}$ can be partitioned into $\ell\geq 1$ intervals $I_1, \ldots, I_\ell$ such that for each $j = 1, \ldots, \ell$ we have $B_{I_j}(v) \stackrel{\Delta}{=} \cup_{s\in I_j} B_s(v)$ satisfies $\sigma|_{B_{I_j}(v)} = \sigma|_{\mathsf B'_{I_j}}$ for some $\mathsf B'_{I_j}$ which is a translated copy of $B_{I_{j}}(v)$ with $\mathsf v_{B_{I_j}(v), \mathsf B'_{I_j}}$'s distinct from each other (and also not equal to the 0-vector).

Before we use the above construction to prove  Proposition~\ref{prop-Step-3}, we derive some preliminary results as preparation. By a similar argument as in Lemma~\ref{lem-unique-typical}, we obtain that
\begin{equation}\label{eq-unique-not-local}
\P(\exists B, B'\in \mathfrak B_{r-1}: \min_{u\in B, u'\in B'}|u - u'|_{\infty} \leq 4r \mbox{ and } \sigma|_B = \sigma|_{B'}) \leq n^{-\epsilon/4}\,.
\end{equation}
The next lemma will also be useful.
\begin{lemma}\label{lem-two-intersecting-pairs}
With probability at least $1-2n^{- \epsilon/4}$ the following holds for all $A, A', B, B'\in \mathfrak B_{r-1}$: 
\begin{equation}\label{eq-two-intersecting-pair}
\mbox{if }A \cap B \neq  \emptyset, A' \cap B' \neq \emptyset \mbox{ and } \mathsf v_{A, B} \neq \mathsf v_{A', B'}, \mbox{ then either }\sigma|_A  \neq \sigma|_{A'} \mbox{ or }\sigma|_{B} \neq \sigma|_{B'}.
\end{equation}
\end{lemma}
\begin{proof}
We first show that if $\min_{u\in A, u'\in A'} |u-u'|_{\infty} > 4r$ and if $A \cap B \neq  \emptyset$ and $A' \cap B' \neq \emptyset$ and in addition $\mathsf v_{A, B} \neq \mathsf v_{A', B'}$, then we have
\begin{equation}\label{eq-two-identical-copies}
\P(\sigma|_A  = \sigma|_{A'} \mbox{ and }\sigma|_{B} = \sigma|_{B'})=q^{-2(r-1)^d}\,.
\end{equation}
To see this, we count the number of labeling configurations on $A \cup A'\cup B\cup B'$ such that $\sigma|_A = \sigma|_{A'}$ and $\sigma|_{B} = \sigma|_{B'}$. Consider a graph with vertex set $A \cup A'\cup B\cup B'$ where the edge set is $\{(u, u + \mathsf v_{A, A'}): u\in A\} \cup \{(u, u' + \mathsf v_{B, B'}): u\in B\}$. Since we assumed $\min_{u\in A, u'\in A'} |u-u'|_{\infty} > 4r$, this graph is a bipartite graph between $A \cup B$ and $A'\cup B'$, and we claim that there is no cycle (nor multiple edge) in this graph. Otherwise, for some $j \geq 1$ there exist distinct $u_1, \ldots, u_j \in A \cup B$ and $u'_1, \ldots, u'_j \in A' \cup B'$ such that
$$u_i + \mathsf v_{A, A'} = u'_i \mbox{ and } u_{i+1} + \mathsf v_{B, B'} =  u'_{i} \mbox{ for } 1\leq i\leq j\,,$$
where we used the convention that $u_{j+1} = u_1$. This implies that $\mathsf v_{A, A'} = \mathsf v_{B, B'}$, arriving at a contradiction.
Therefore, each edge in this graph reduces the number of valid labeling configurations (i.e., those satisfy $\sigma|_A = \sigma|_{A'}$ and $\sigma|_B = \sigma|_{B'}$) by a factor of $q$. This implies \eqref{eq-two-identical-copies} since the number of edges is $2 (r-1)^d$.
Since the number of choices for $A, A', B, B'$ with $A \cap B\neq \emptyset$ and $A'\cap B' \neq \emptyset$ is at most $n^{2d}(2r)^{2d}$, we can apply a union bound and obtain that with probability at least $1 - n^{-\epsilon/4}$ for all $A, A', B, B'\in \mathfrak B_{r-1}$ with $\min_{u\in A, u'\in A'} |u-u'|_{\infty} > 4r$ we have \eqref{eq-two-identical-copies}. Combined with \eqref{eq-unique-not-local}, this implies the lemma.
\end{proof}
We now come back to the proof of Proposition~\ref{prop-Step-3}. Without loss of generality we assume that the event in \eqref{eq-unique-not-local} and in the lemma statement of Lemma~\ref{lem-two-intersecting-pairs}. Thus, it suffices to consider $B'_{I_j}$'s (which are possible realizations of $\mathsf B'_{I_j}$'s) such that
\begin{equation}\label{eq-disjointness-B-I-j}
B_{I_j}(v)\cap  B'_{I_{j'}} = \emptyset \mbox{ for all } 1\leq j, j' \leq \ell  \mbox{ and } B'_{I_j} \cap B'_{I_{j'}} = \emptyset \mbox{ for } 1\leq j\neq j' \leq \ell\,.
\end{equation}
We next count the number of labeling configurations on $\cup_{j=1}^\ell B_{I_j}(v) \cup B'_{I_j}$ such that $\sigma|_{B_{I_j}(v)} = \sigma|_{B'_{I_j}}$ for $1\leq j\leq \ell$. Consider a graph with vertex set  $\cup_{j=1}^\ell B_{I_j}(v) \cup B'_{I_j}$ and  edge set $\mathcal E = \cup_{j=1}^\ell\{(u, u + \mathsf v_{B_{I_j}(v), B'_{I_j}}: u\in B_{I_j}(v))\}$. By \eqref{eq-disjointness-B-I-j}, we see that for each $u\in  B'_{I_j}$ there is a single edge incident to $u - \mathsf v_{B_{I_j}(v), B'_{I_j}}$, i.e., the edge between $u$ and $u - \mathsf v_{B_{I_j}(v), B'_{I_j}}$. As a result, in this graph there is no cycle (or multiple edge). That is to say, each edge in this graph reduces the number of configurations by a factor of $q$. Since 
$$|\mathcal E| =\sum_{i=1}^\ell (r-1)r^{d-2}(r+|I_i| - 1)\geq (r-1)^{d}(1 + \ell)\,,$$
we see that 
$$\P(\sigma|_{B_{I_j}(v)} = \sigma|_{B'_{I_j}} \mbox{ for } 1\leq j\leq \ell) \leq q^{-(r-1)^{d}(1 + \ell)}\,.$$

Summing over $v\in \Lambda_n$,  $1 \le \ell \le r$,  all partitions of $I_1,\ldots,I_\ell$ and all choices of $B'_{I_j}$ for $1 \le j \le \ell$, we derive that
\begin{align*}
\P(\exists v\in \Lambda_n: v \mbox{ is undetermined}) &\leq \sum_{v\in \Lambda_n} \sum_{1\leq \ell \leq r} \sum_{I_1, \ldots, I_\ell} \sum_{B'_{I_1}, \ldots, B'_{I_\ell}} \P(\sigma|_{B_{I_j}(v)} = \sigma|_{B'_{I_j}} \mbox{ for } 1\leq j\leq \ell)\\
&\leq n^d r \binom{r-1}{\ell-1} n^{d\ell} q^{-(r-1)^{d}(1 + \ell)} = o(1)\,,
\end{align*}
completing the proof of Proposition~\ref{prop-Step-3}.

\section{Minor modifications for rotation and reflection symmetry} \label{sec:extension}

In this section, we briefly discuss how to extend our proof to obtain the same result as in Theorem~\ref{thm-main} when rotation and reflection symmetry is taken into account, as mentioned in Remark~\ref{rem-extension}. More precisely, we say a configuration $\sigma$ on $\Lambda_n$ is isomorphic to $\sigma'$ if there exists a map $\tau$ that is a composition of rotation and reflection of $\Lambda_n$ such that $\sigma = \sigma'(\tau)$. In this case, the local observations are given up to isomorphism with respect to an $r$-box, and our goal is to recover $\sigma$ on $\Lambda_n$ up to isomorphism with respect to $\Lambda_n$.

For the proof of non-identifiability, we must ensure that $\sigma'$ is not isomorphic to $\sigma$ where $\sigma'$ is obtained from swapping some labels  in $\sigma$ as we described in our proofs. In the case of $d=1$, we can partition $\Lambda_n$ into $8$ disjoint and consecutive intervals $I_1,\ldots,I_8$ with length $\lfloor n/8 \rfloor$, and perform a swapping with $I_2, \ldots, I_7$ replacing $I_1, \ldots, I_6$ as in Section~\ref{sec:1d}. We can then use $I_1,I_8$ to guarantee that the $\sigma'$ obtained from our swapping operation is not isomorphic to $\sigma$ (since with probability tending to 1 we have that $\sigma|_{I_1}$ is not isomorphic to $\sigma|_{I_8}$). In the case for $d \ge 2$, the proof of Proposition \ref{prop-match-empirical-profile} still works with the following additional property: with probability tending to 1 there is no rotation/reflection so that the labels that are not $\mathsf 1$ nor $2$ are preserved (this can be checked easily).

In the identifiable regime, the extension of our recovery procedure is a little more complicated as we next explain.
For each box $B$, we say $B$ has an automorphism if there is a composition of rotation and reflection which is non-identical and maps $\sigma|_B$ to itself. For each $B \in \mathfrak B_s$, we modify the definition of open such that $B$ is open if each $B' \in \mathfrak B_{r-1}$ that is contained in $B$ is unique and does not have an automorphism. The additional condition on automorphism ensures that Lemma \ref{lem-successively-determine} still holds since once a unique $(r-1)$-box $A$ without automorphism is determined, the vertices neighboring to $A$ can also be determined. Moreover, the probability that an $(r-1)$-box has an automorphism is at most $2^d q^{-(r-1)^d/2}= O(n^{-d/2})$, which is substantially smaller than the probability of being non-unique, so all of our probabilistic estimates remain valid.

We also need to modify Proposition \ref{prop-Step-1} and we can do it by using the following fact: on the one hand, if the labeling configuration of $A \in \mathfrak B_{r-1}$ appears only once as $\sigma|_{B'}$ where $B'$ is an $(r-1)$-box in some $r$-box $B$ (note that this event is measurable with respect to  $\{\sigma|_{B}: B\in \mathfrak B_{n, r}\}$), then $A$ must lie on the corner of $\Lambda_n$; on the other hand, if $A \in \mathfrak B_{r-1}$ is contained in a $2r$-box on a corner of $\Lambda_n$ and if this $(2r)$-box is open, then  the labeling configuration of $A$ appears only once as $\sigma|_{B'}$ where $B'$ is an $(r-1)$-box in some $r$-box $B$. Assuming that $\Lambda_{2r}$ is open, this fact ensures that in any possible labeling of $\Lambda_n$ with $\{\sigma|_{B}: B\in \mathfrak B_{n, r}\}$, the labeling configuration $\sigma|_{\Lambda_{r-1}}$ must appear on the corner of $\Lambda_n$. Since we only care about the labelings up to isomorphism, we can choose an arbitrary corner and put the labeling configuration $\sigma|_{\Lambda_{r-1}}$ there.

The proofs of Lemma \ref{lem-S-closed}, Proposition \ref{prop-Step-2}, Lemma \ref{lem-two-intersecting-pairs} and Proposition \ref{prop-Step-3} will be roughly same, except that we should choose the phantom box (i.e., the box that has the same labeling configuration as another box) together with its orientation when considering an $(r-1)$-box as non-unique. In view of this, we can define $\mathsf u_{A,A'}$ as the transformation in $\R^d$ that maps $A$ to $A'$, if $A,A' \subset \Lambda_n$ and $A$ is congruent to $A'$. In particular, if $A,A' \in \mathfrak B_s$, then $\mathsf u_{A,A'}$ can be seen as the composition of $\mathsf v_{A,A'}$ and an reflection/rotation that preserves $A'$. The analysis of the reduction of a factor $q$ for the enumeration of valid labeling configurations would be the same by replacing $\mathsf v_{A,A'}$ with $\mathsf u_{A,A'}$ (where instead of adding $\mathsf v_{A,A'}$ for the translation, we replace the transformation by the map $\mathsf u_{A,A'}$). Finally, all possible orientations together only contribute a multiplicative factor of $2^d$ to the probability for each box being non-unique, and as a result all of our probabilistic estimates remain valid.

\end{document}